\documentclass[pdflatex,sn-mathphys-num]{sn-jnl}

\usepackage{graphicx}%
\usepackage{multirow}%
\usepackage{amsmath,amssymb,amsfonts}%
\usepackage{amsthm}%
\usepackage{mathrsfs}%
\usepackage[title]{appendix}%
\usepackage{xcolor}%
\usepackage{textcomp}%
\usepackage{manyfoot}%
\usepackage{booktabs}%
\usepackage{algorithm}%
\usepackage{algorithmicx}%
\usepackage{algpseudocode}%
\usepackage{listings}%

\theoremstyle{thmstyleone}%
\newtheorem{theorem}{Theorem}

\theoremstyle{thmstyleone}%
\newtheorem{lemma}{Lemma}%

\newtheorem{proposition}[theorem]{Proposition}%

\theoremstyle{thmstyletwo}%
\newtheorem{remark}{Remark}%

\theoremstyle{thmstylethree}%

\begin{document}

\title[Article Title]{Rate estimates for weighted total variation norm in terms of Wasserstein distances
}


\author*[1,2]{\fnm{Iván} \sur{Ivkovic}}\email{ivkovic@renyi.hu}

\author*[1,2]{\fnm{Miklós} \sur{Rásonyi}}\email{rasonyi@renyi.hu}

\affil[1]{\orgname{HUN-REN Alfréd Rényi Institute of Mathematics}, \orgaddress{\street{Reáltanoda street 13-15}, \city{Budapest}, \postcode{1053}, \state{Budapest}, \country{Hungary}}}

\affil[2]{\orgdiv{Institute of Mathematics}, \orgname{Faculty of Science, Eötvös Loránd University}, \orgaddress{\street{ Pázmány Péter sétány 1/C}, \city{Budapest}, \postcode{1117}, \state{Budapest}, \country{Hungary}}}


\abstract{We study the weighted total variation distance between probability measures. Using Fourier-analytic tools, we present estimates
in terms of Wasserstein distances between the respective probabilities, under appropriate smoothness and moment conditions. 
We apply our results to convergence of functionals in Malliavin calculus.}

\keywords{Wasserstein metric, weighted total variation distance, Malliavin calculus, smooth densities \newline  
\noindent\textbf{MSC 2020:} 60B10; 60H07; 42B10}

\maketitle
{}

\noindent\textbf{Acknowledgments.} {Supported by the Hungarian National Research, Development and Innovation Office within 
the framework of the Thematic Excellence Program 2021; National Research subprogramme ``Artificial intelligence, 
large networks, data security: mathematical foundation and applications'' and also by the grant K 143529.

Project no. C2290224 has been implemented with the support provided by the Ministry of Culture and Innovation of Hungary from the National Research, Development and Innovation Fund, financed under the KDP-2023 funding scheme.}

\section{Introduction}

Fix an integer $d\geq 1$. Let $\mathcal{P}$ denote the family of probabilities on the Borel sigma-algebra
of $\mathbb{R}^{d}$. For a measurable function  $V:\mathbb{R}^{d}\to \mathbb{R}_{+}$, define the $V$-weighted total variation distance
of $\mu,\nu\in\mathcal{P}$ as
$$
d_{V}(\mu,\nu):=\int_{\mathbb{R}^{d}}V(x)|||\mu-\nu|||(dx),
$$
where $|||\mu-\nu|||$ is the total variation measure of $\mu-\nu$, for more information see e.g. Appendix D in \cite{eric}. 
The notation $|\cdot|$ will stand for the Euclidean norm in $\mathbb{R}^{k}$ or $\mathbb{C}^{k}$ for some $k$, which will always be
clear from the context (in most cases, $k=d$).  
In the present article we are dealing with
$V_{p}(x):=1+|x|^{p}$, $x\in\mathbb{R}^{d}$ for some $p\geq 1$, we denote the corresponding distance by $\rho_{p}(\mu,\nu):=d_{V_{p}}(\mu,\nu)$.

Such distances play a prominent role in Markov process theory and in related probabilistic models.
On uncountable state spaces standard convergence guarantees are often formulated in terms of $d_{V}$, where $V$ is a Lyapunov function
of the given Markov process, see \cite{meyn-tweedie,hm,eric}.

For $p\geq 1$, let us now define also the $p$-Wasserstein distances
$$
W_{p}(\mu,\nu):=\left(\inf_{\pi\in\mathcal{C}(\mu,\nu)}\int_{\mathbb{R}^{d}}\int_{\mathbb{R}^{d}}|x-y|^{p}\pi(dx,dy)\right)^{1/p},\ \mu,\nu\in\mathcal{P}.
$$
Here $\mathcal{C}(\mu,\nu)$ is the set of probabilities on $\mathbb{R}^{d}\times\mathbb{R}^{d}$ such that
their respective marginals are $\mu$ and $\nu$, see \cite{villani}.

Wasserstein distances and optimal transport problems have risen to great popularity in probability theory. 
Without giving a detailed overview we only remark that, satisfying demands from machine learning applications, recent convergence guarantees 
for Markov processes are often formulated in terms of these metrics both in discrete and continuous time, see the studies \cite{zimmer,eberle,majka}. 

It is known that $W_{p}\leq C_{p}\rho^{1/p}_{p}$ with some constant $C_{p}$, see Theorem 6.15 of \cite{villani},
so convergence in $\rho_{p}$ is (much) stronger than in $W_{p}$, in general.
In certain situations, however, convergence in $W_{q}$ for some $q$ may imply convergence in $\rho_{p}$. 
Following the study \cite{rate}, where total variation rates have been deduced from rates in the Fortet-Mourier metric (and similar
metrics), here we seek to estimate the weighted total variation distance between random variables in terms of
Wasserstein distances: in Lemma \ref{fourier1} below we establish that, under smoothness and moment conditions, $\rho_{p}$
convergence rates can be obtained from $W_{q}$ ones. More stringent hypotheses
lead to rates for $\rho_{p}$ that are even better, see Lemma \ref{main2}. Our proof is based on
Fourier-analytic arguments inspired by those in \cite{rate}, but technically more complicated. 

Our interest in these results has been partially driven by a branch of stochastic analysis: Malliavin calulus,
which is a sort of differential calculus for random variables on the Wiener space (in other words, for
functionals of Brownian motion), see \cite{nualart,bally}. It is an efficient tool for studying various
stochastic differential equations, in particular, to establish that the related random variables have (smooth and positive)
densities, light tails, etc. Recent work established that, for certain sequences of Malliavin differentiable functionals,
convergence in the (a priori much weaker) Fortet-Mourier distance implies convergence in the stronger total variation distance,
see \cite{bc1,bally,bc2}. In the present article we continue research into this direction, establishing in Theorem \ref{smoothsequence2}, for the first time, that 
convergence in $q$-Wasserstein distance for some $q>1$ implies (for smooth, non-degenerate functionals of Malliavin calculus) convergence in 
the weighted total variation norms $\rho_{p}$ as well.

After stating rigorously our main results (Lemmas \ref{fourier1} and \ref{main2}) in Section \ref{sec2}, they are proved in
Section \ref{sec3} and applied to Malliavin calculus in Section \ref{sec4}.

\section{Estimates in weighted total variation norm}\label{sec2}

We denote by $\langle\cdot,\cdot\rangle$ 
the usual scalar product in $\mathbb{R}^{d}$. For a function $g:\mathbb{R}^{d}\to\mathbb{C}^{k}$ we
define its supremum norm as
$$
||g||_{\infty}:=\sup_{x\in\mathbb{R}^{d}}|g(x)|,
$$
whatever $k$ is. 
Let $g:\mathbb{R}^{d}\to \mathbb{C}$, then for any $l\geq 1$, $g^{(l)}:\mathbb{R}^{d}\to\mathbb{C}^{d^{l}}$ refers to the $l$th 
derivative of the function $g=g^{(0)}$. $C^{\infty}$ denotes the family of infinitely many times differentiable functions
from $\mathbb{R}^{d}$ to $\mathbb{C}$.

Let $(\Omega,\mathcal{F},P)$ be a probability space, expectation will be denoted by $E\left[\cdot\right]$.{}
For an $\mathbb{R}^{d}$-valued random 
variable $X$, $\mathcal{L}(X)$
denotes its law. 

The characteristic function of an $\mathbb{R}^{d}$-valued random variable $\xi$ is defined as
$$
\phi_{\xi}(u):=E[e^{i\langle u,\xi\rangle}],\ u\in\mathbb{R}^{d},
$$
this is (constant times) the inverse Fourier-transform of the probability measure $\mathcal{L}(\xi)$, see \eqref{maci} below
for the convention we use for Fourier-transform.
Here and in the sequel $i\in\mathbb{C}$ refers to the imaginary unit.

In Corollary 2.2 of \cite{rate} total variation distance was estimated in terms of the Fortet-Mourier
metric, under suitable assumptions on smoothness and the existence of moments. 
The main result of this paper compares weighted total variation distance of two random variables to their
Wasserstein distances, under similar assumptions. 

\begin{lemma}\label{fourier1}
Let $\xi,\eta$ be $d$-dimensional random variables with density functions 
(with respect to the $d$-dimensional Lebesgue measure) denoted by
$f_{\xi},f_{\eta}$. Assume that $f_{\xi},f_{\eta}\in C^{\infty}$ and
\begin{equation}\label{intim}
\Vert f^{(k)}_{\xi}(x) (1+|x|)^{l}\Vert_{\infty}+\Vert f^{(k)}_{\eta}(x) (1+|x|)^{l}\Vert_{\infty}\leq d_{k,l},
\end{equation}
holds for all $k,l\in\mathbb{N}$, where $d_{k,l}$ are a double sequence of (finite) constants.
Then for all $q>1$, $0<\epsilon<1$ and $p\geq 1$ there is a constant $C$
(which depends only on the sequence $d_{k,l}$, $p$, $q$, $d$ and $\epsilon$ but not on $\xi,\eta$)
such that
\begin{equation}\label{weighted}
\rho_{p}(\mathcal{L}(\xi),\mathcal{L}(\eta))\leq C W_{q}(\mathcal{L}(\xi),\mathcal{L}(\eta))^{1-\epsilon}.
\end{equation}
Furthermore, for each $p\geq 1$ and multiindex $\alpha$,{}
\begin{equation}\label{infi}
||(\partial_{\alpha}f_{\xi}(x)-\partial_{\alpha}f_{\eta}(x))(1+|x|^{p})||_{\infty}\leq C' W_{q}(\mathcal{L}(\xi),\mathcal{L}(\eta))^{1-\epsilon}
\end{equation}
holds for another constant $C'$ (which depends only on the sequence $d_{k,l}$, $p$, $q$, $d$, $|\alpha|$ and $\epsilon$ but not on $\xi,\eta$). 
\end{lemma}

Even stronger hypotheses lead to stronger conclusions.

\begin{lemma}\label{main2} 
Let $\xi,\eta$ be $d$-dimensional random variables with 
$$E[e^{r|\xi|}+e^{r|\eta|}]\leq C^{\sharp}<\infty,{}
$$
for some $r>0$. Let their characteristic functions be
$\phi_{\xi},\phi_{\eta}$. 
Assume that $\phi_{\xi},\phi_{\eta}\in C^{\infty}$ and
\begin{equation}\label{intim2}
\int_{\mathbb{R}^{d}}|\phi^{(k)}_{\xi}(x)| e^{r_{k}|x|}\, dx+ \int_{\mathbb{R}^{d}}
|\phi^{(k)}_{\eta}(x)| e^{r_{k}|x|}\, dx\leq c_{k},
\end{equation}
holds for all $k$, where $c_{k}, r_{k}>0$ are sequences of (finite) constants.
Then for all $q>1$, and $p\geq 1$ there is a constant $C$
(which depends only on the sequences $c_{k}$, $r_{k}$, on $p$, $q$, $d$, $r$, $C^{\sharp}$ and $\epsilon$ but not on $\xi,\eta$)
such that
\begin{equation}\label{weighted2}
\rho_{p}(\mathcal{L}(\xi),\mathcal{L}(\eta))\leq CW_{q}(\mathcal{L}(\xi),\mathcal{L}(\eta))|\ln(W_{q}(\mathcal{L}(\xi),\mathcal{L}(\eta))|^{2d+1}.
\end{equation}
\end{lemma}

\begin{remark}{\rm We do not know whether and how the previous statements could be proved in the case $q=1$.}
\end{remark}

\begin{remark}{\rm Proposition \ref{equiv1} below shows that, instead of \eqref{intim}, in Lemma \ref{fourier1} we may require
\begin{equation}\label{smo}
\Vert |\phi^{k}(x)|(1+|x|)^{l}\Vert_{\infty}\leq b_{k,l},\ k,l\in\mathbb{N}	
\end{equation}
for $\phi:=\phi_{\xi}$ and $\phi=\phi_{\eta}$ with a double sequence of positive
constants $b_{k,l}$. 

Sometimes working with characteristic functions is easier. We sketch a very simple example when $d=1$. Let $\xi_{n}$ be a
sequence with $\sup_{n}E|\xi_{n}|^{k}<\infty$ for each $k\in\mathbb{N}$, which implies that $\phi_{\xi_{n}}$ is
$C^{\infty}$ for all $n$ with $|\phi_{\xi_{n}}^{(k)}(u)|$ uniformly bounded in $n$, for each $k$. Assume that $h_{n}=E^{1/q}|\xi_{n}-\xi|^{q}\to 0$
for some $q>1$. Consider the ``smoothed out'' random variables $\eta_{n}:=\xi_{n}+\vartheta$ where $\vartheta${}
is independent of $\sigma(\xi,\xi_{n},n\in\mathbb{N})$ and $\phi_{\vartheta}$ satisfies \eqref{smo}. We claim that,
for every $\epsilon>0$, $\rho_{p}(\mathcal{L}(\eta_{n}),\mathcal{L}(\xi+\vartheta))=O(h_{n}^{1-\epsilon})$.
Indeed, we only need to check that $\phi_{\eta_{n}}$ satisfy \eqref{smo} uniformly in $n$ which is easy: for any multiindex,
$\partial_{\alpha}(\phi_{\eta_{n}})=\partial_{\alpha}(\phi_{\xi_{n}}\phi_{\vartheta})$ is a sum of products of terms of
the form either $\phi_{\xi_{n}}^{(k)}$ or $\phi_{\vartheta}^{(l)}$ for some $k,l$ and at least one term is $\phi_{\vartheta}^{(l)}$. 
Since $\phi_{\vartheta}^{(l)}(u)$ decreases faster than any polynomial as $|u|\to\infty$, so does 
$\partial_{\alpha}(\phi_{\xi_{n}}\phi_{\vartheta})(u)$, uniformly in $n$. Lemma \ref{fourier1} now applies.

Smoothing the sequence via convolution guarantees rapid decay and uniform smoothness of the 
characteristic functions in the frequency domain, thereby enabling convergence rate estimates in 
strong distances under only $L^q$-convergence of the original sequence.
}	
\end{remark}

\section{Proof of main results}\label{sec3}

For each (Lebesgue-) integrable $f:\mathbb{R}^{d}\to\mathbb{C}$, denote by $\hat{f}(u)$ its Fourier transform, given by
\begin{equation}\label{maci}
\hat{f}(u) = \int_{\mathbb{R}^{d}} f(x) e^{-i\langle u, x \rangle} \, dx,\quad u\in\mathbb{R}^{d}.
\end{equation}
Note that $\hat{f}^{(k)}$ is the $k$th derivative of the Fourier-transform of $f$ while
$\widehat{f^{(k)}}$ is the Fourier-transform of the $k$th derivative of $f$.

For a multiindex $\alpha=(\alpha_{1},\ldots,\alpha_{d})\in\mathbb{N}^{d}$, $\partial_{\alpha}$ refers to the respective
partial derivative, the \emph{multiplicity} of the multiindex is $|\alpha|:=\sum_{j=1}^{d}\alpha_{j}$.

It is well-known that the space of rapidly decreasing smooth functions is closed under the Fourier transform.
The following result formulates this fact in a slightly more precise way, which will be necessary for our later arguments.

\begin{proposition}\label{equiv1}
Let $\mathbf{F}$ be a set of functions $f:\mathbb{R}^{d}\to\mathbb{C}$.
The following two statements are equivalent:

\begin{enumerate}

\item $\mathbf{F}\subset C^{\infty}$, there exists a double sequence of positive
constants $b_{k,l}$ such that for all $k,l\in\mathbb{N}$, 
\begin{equation}\label{el1}
\sup_{f\in\mathbf{F}}\Vert f^{(k)}(x)(1+|x|)^{l}\, \Vert\leq b_{k,l}.	
\end{equation}

\item All $f\in\mathbf{F}$ are Lebesgue-integrable with $\hat{f}\in C^{\infty}$; there exists a double sequence of positive
constants $d_{k,l}$ such that for all $k,l\in\mathbb{N}$, 
$$
\sup_{f\in\mathbf{F}}||\hat{f}^{(k)}(u) (1+|u|)^{l}||_{\infty}\leq d_{k,l}.	
$$
\end{enumerate} 
\end{proposition}
\begin{proof}
We are showing $1.\to 2.$

Theorem 1.7 of \cite{fourier} and \eqref{el1} imply that $\hat{f}\in C^{\infty}$ and, for all multiindices $\alpha$,{}
\begin{equation*}
\partial_{\alpha}\hat{f}(u) = (-i)^{|\alpha|} \int_{\mathbb{R}^{d}} x_{1}^{\alpha_{1}}\cdots x_{d}^{\alpha_{d}} f(x) e^{-i\langle u, x \rangle} \, dx.
\end{equation*}

Next, fix $m\in \{1,\ldots,d\}$ and $l\in\mathbb{N}$. The notation $\partial_{m}^{l}$ refers to taking $l$ times the partial
derivative with respect to the $m$th variable. We will establish by induction that $\partial_{\alpha}\hat{f}(u) u_{m}^{l} i^{l}$ is the 
Fourier transform of $(-i)^{|\alpha|}\partial^{l}_{m}[f(x)x_{1}^{\alpha_{1}}\cdots x_{d}^{\alpha_{d}}]$. The case $l=0$ being already done
we proceed to the induction step. 
Let $e\in\mathbb{R}^{d}$ be such that $e_{m}=1$ and $e_{j}=0$ for $j\neq m$, define
\[
g_{e}(h,x) := \frac{(-i)^{|\alpha|} \partial_{m}^{l} [f(z)z_{1}^{\alpha_{1}}\cdots z_{d}^{\alpha_{d}}]\vert_{z=x+he}-
(-i)^{|\alpha|} \partial_{m}^{l}[f(x)x_{1}^{\alpha_{1}}\cdots x_{d}^{\alpha_{d}}]}{h}.
\]
As $f\in C^{\infty}$, $g_{e}(h,x)\to (-i)^{|\alpha|}\partial_{m}^{l+1}[f(x)x_{1}^{\alpha_{1}}\cdots x_{d}^{\alpha_{d}}]$
pointwise as $h\to 0$. We need to show that this convergence takes place in the $L^{1}$-norm, that is,
\begin{equation}\label{hello}
\int_{\mathbb{R}^{d}}\left| g_{e}(h,x)-
(-i)^{|\alpha|}\partial_{m}^{l+1}[f(x)x_{1}^{\alpha_{1}}\cdots x_{d}^{\alpha_{d}}] \right|\, dx\to 0
\end{equation}
as $h\to 0$. Once this has been shown, we may invoke Theorem 1.8 of \cite{fourier} to conclude that 
$\partial_{\alpha}\hat{f}(u) u_{m}^{l+1}i^{l+1}$ is indeed the 
Fourier transform of $(-i)^{|\alpha|}\partial^{l+1}_{m}[f(x)x_{1}^{\alpha_{1}}\cdots x_{d}^{\alpha_{d}}]$.
By dominated convergence, it suffices to present an integrable majorant of the integrand in \eqref{hello}. 

For $|h|<1$, by the Lagrange mean value theorem, there is $|\chi(x,h)|<1$ such that
$$
g_{e}(h,x)=
(-i)^{|\alpha|}\partial_{m}^{l+1}[f(z)z_{1}^{\alpha_{1}}\cdots z_{d}^{\alpha_{d}}]\vert_{z=x+\chi(x,h)e}.
$$
Notice that $\partial_{m}^{l+1}[f(z)z_{1}^{\alpha_{1}}\cdots z_{d}^{\alpha_{d}}]$ is a sum of polynomial functions
multiplied by derivatives of $f$ hence, by \eqref{el1}, there is a constant $C^{\flat}$ such that the integrand
in \eqref{hello} is smaller than
$$
\frac{C^{\flat}}{1+|x+\chi(x,h)e|^{d+1}}\leq \frac{C^{\natural}}{(1+|x|)^{d+1}}
$$
for all $x$ with some constant $C^{\natural}$, implying \eqref{hello}.

Define
\[
\check{C}_{k,l}:=\sup_{f\in\mathbf{F}}\max_{|\alpha|\leq k}\left\{\max_{1\leq m\leq d} \int_{\mathbb{R}^{d}} |\partial_{{m}}^{l}[f(x)
x_{1}^{\alpha_{1}}\cdots x_{d}^{\alpha_{d}}]|\, dx\right\},
\]
this is finite by \eqref{el1}.
For $|u|<1$, we use the bound
\begin{equation}\label{bound_small}
|\partial_{\alpha}\hat{f}(u)|\leq \check{C}_{|\alpha|,0}\leq \check{C}_{|\alpha|,l}.
\end{equation} 
For $|u|\geq 1$,
\[
|u_{m}|^{l}|\partial_{\alpha}\hat{f}(u)|\leq \check{C}_{|\alpha|,l}
\]
since, as already shown, $\partial_{\alpha}\hat{f}(u) u_{m}^{l}i^{l}$ is  the 
Fourier transform of $(-i)^{|\alpha|}\partial^{l}_{m}[f(x)x_{1}^{\alpha_{1}}\cdots x_{d}^{\alpha_{d}}]$. 
For each $u$, there always exists $m^{*}=m^{*}(u)$ such that $|u_{m^{*}}|\geq |u|/\sqrt{d}$. Therefore,
\[
|\partial_{\alpha}\hat{f}(u)|\leq \sqrt{d}\check{C}_{|\alpha|,l}/|u|^{l}\leq 
\frac{\sqrt{d}\check{C}_{|\alpha|,l}2^{l}}{(1+|u|)^{l}},
\quad |u|\geq 1.
\]
From \eqref{bound_small}, we also obtain
\[
|\partial_{\alpha}\hat{f}(u)|\leq \frac{\sqrt{d}\check{C}_{|\alpha|,l}2^{l}}{(1+|u|)^{l}},
\quad |u|< 1.
\]
This completes the proof of the direction $1. \xrightarrow{} 2.$ by setting $d_{k,l}:=\check{C}_{k,l}\sqrt{d}2^{l}$.

The implication \( 2. \to 1. \) follows by applying the argument used in the proof of \( 1. \to 2. \) to the inverse Fourier transform. 
\end{proof}

\begin{proof}[Proof of Lemma \ref{fourier1}]
By Theorem 4.1 of \cite{villani} there is a coupling $(\xi',\eta')${}
with $\mathcal{L}(\xi)=\mathcal{L}(\xi')$ and $\mathcal{L}(\eta)=\mathcal{L}(\eta')${}
such that $W_{q}(\xi,\eta)=E^{1/q}[|\xi'-\eta'|^{q}]$. From this moment
on we assume $\xi:=\xi'$, $\eta=\eta'$. 
To simplify notation, we write $A:=E^{1/q}[|\xi-\eta|^{q}]$. The case $A=0$ is trivial, we
assume $0<A\leq 1$ for the moment. 

Denote by $\phi_{\xi}(u):=E[e^{i\langle u,\xi\rangle}]$, $u\in\mathbb{R}^{d}$ 
and $\phi_{\eta}$ the respective characteristic 
functions of $\xi,\eta$. 

From \eqref{intim} and from the implication $2.\to 1.$ in Proposition \ref{equiv1} we know 
that for each $k,l\in\mathbb{N}$ there is $d_{k,l}$ such that for each multiindex $\alpha$,
\begin{equation}\label{kiralyfi}
|\partial_{\alpha}\phi_{\xi}(u)|+|\partial_{\alpha}\phi_{\eta}(u)|\leq{}
\frac{d_{|\alpha|,l}}{(1+|u|)^{l}},\ u\in\mathbb{R}^{d}.
\end{equation}

Assume first $|\alpha|\geq 1$. By applying the product difference bound,
\begin{eqnarray*}
& &\left| \xi_1^{\alpha_1} \dots \xi_d^{\alpha_d} - \eta_1^{\alpha_1} \dots \eta_d^{\alpha_d} \right| \\
& \leq & |\alpha| \max_j |\xi_j - \eta_j|
\left[ \max_j |\xi_j| + \max_j |\eta_j| \right]^{|\alpha|-1} \\
& \leq &
 |\alpha| |\xi - \eta| 2^{|\alpha|-1} (|\xi|^{|\alpha|-1} + |\eta|^{|\alpha|-1}) \\
\end{eqnarray*}
Next we use the elementary inequality $|e^{ix}-e^{iy}|\leq |x-y|$, $x,y\in\mathbb{R}$ for obtaining the following bound
\begin{eqnarray*}
    & & |e^{i\langle u,\xi\rangle}\xi_{1}^{\alpha_{1}}\ldots \xi_{d}^{\alpha_{d}}-
e^{i\langle u,\eta\rangle}\eta_{1}^{\alpha_{1}}\ldots \eta_{d}^{\alpha_{d}}|\\
&\leq& 	|(e^{i\langle u,\xi\rangle}-e^{i\langle u,\eta\rangle})\xi_{1}^{\alpha_{1}}\ldots \xi_{d}^{\alpha_{d}}|
+|\xi_{1}^{\alpha_{1}}\ldots \xi_{d}^{\alpha_{d}}-\eta_{1}^{\alpha_{1}}\ldots\eta_{d}^{\alpha_{d}}|\\
&\leq& |\xi|^{|\alpha|}|\langle u,\xi-\eta\rangle|+|\alpha| 2^{|\alpha|-1} (|\xi|^{|\alpha|-1} + |\eta|^{|\alpha|-1}) |\xi - \eta|.
\end{eqnarray*}
The inequalities \eqref{intim} guarantee that all the moments of $\xi$ are finite hence one can ``differentiate under the expectation'' in
the definition of $\phi_{\xi}$ and arrive at $$
\partial_{\alpha}\phi_{\xi}(u)=i^{|\alpha|}E[\xi_{1}^{\alpha_{1}}\ldots \xi_{d}^{\alpha_{d}}e^{i\langle \xi,u\rangle}],$$
analogously for $\phi_{\eta}$. We may thus estimate with $q'$ satisfying $1/q+1/q'=1$,
\begin{eqnarray}\label{fabol}
& & |\partial_{\alpha}\phi_{\xi}(u)-\partial_{\alpha}\phi_{\eta}(u)|\\
\nonumber &\leq& E^{1/q'}[|\xi|^{|\alpha|q'}]  |u| E^{1/q}[|\xi - \eta|^{q}]
	+	|\alpha| 2^{|\alpha|-1}  E^{1/q'}[(|\xi|^{|\alpha|-1} + |\eta|^{|\alpha|-1})^{q'}]E^{1/q}[|\xi - \eta|^{q}] \\
\nonumber   &\leq& C^{\circ}_{|\alpha|} A[|u|+1],
\end{eqnarray}
with constants $C^{\circ}_{k}$, $k\geq 1$, since by \eqref{intim}, all moments of $\xi,\eta$ are finite. In the case $|\alpha|=0$ we get
$$
|\phi_{\xi}(u)-\phi_{\eta}(u)|\leq |u|E|\xi-\eta|\leq |u| A\leq A(|u|+1)
$$
and set $C^{\circ}_{0}:=1$. 

Note that the area of the $d$-dimensional unit sphere is $2\pi^{d/2}/\Gamma(d/2)$.{}
Before proceeding, we estimate, for $k>d$,
\begin{eqnarray}\nonumber
& & \int_{|u|\geq M}(1+|u|)^{-k}\, du\leq \int_{|u|\geq M}|u|^{-k}\, du\\
&=& \int_{r\geq M}\frac{2\pi^{d/2}}{\Gamma(d/2)}r^{d-1}r^{-k}\, dr= 
\frac{\gamma_{k}}{M^{k-d}}	
\label{explin}
\end{eqnarray}
holds for all $M>0$ with constant $\gamma_{k}:=\frac{2\pi^{d/2}}{\Gamma(d/2)(k-d)}$.

We may and will assume from now on that $p\geq 2$ is an even integer. 
Introduce the differential operators $\Delta_{p}:=\sum_{j=1}^{d}\partial_{{j}}^{p}$.
We know from the proof of Proposition \ref{equiv1} that
$$
f_{\xi}(x)\sum_{j=1}^{d} |x_{j}|^{p}=f_{\xi}(x)\sum_{j=1}^{d} x_{j}^{p}=
\frac{(-i)^{p}}{(2\pi)^{d}}\int_{\mathbb{R}^{d}}\Delta_{p}\phi_{\xi}(u)e^{-i\langle u,x\rangle}\, du,
$$ 
and similarly for $f_{\eta}$. 

Note that $\Delta_{p}$ is the sum of $d$ differential operators with multiindices of multiplicity $p$.
Recall that the ball of radius $M$ in $\mathbb{R}^{d}$ has volume $M^{d}\pi^{d/2}/\Gamma\left(\frac{d}{2} + 1\right)$.
Hence for each $x\in\mathbb{R}$ and for each $l>d$
\begin{eqnarray}\label{majkell}
& & |f_{\xi}(x)-f_{\eta}(x)| \left[\sum_{j=1}^{d}|x_{j}|^{p}\right] 
\\ &\leq&\nonumber \frac{1}{(2\pi)^{d}}\int_{|u|\leq M}|\Delta_{p}\phi_{\xi}(u)-\Delta_{p}\phi_{\eta}(u)|\, du
+\frac{1}{(2\pi)^{d}}\int_{|u|>M} (|\Delta_{p}\phi_{\xi}(u)|+|\Delta_{p}\phi_{\eta}(u)|)\, du\\  &\leq& \nonumber
\frac{C^{\circ}_p A d \pi^{d/2}}{(2\pi)^d \Gamma\left(\frac{d}{2} + 1\right)} \left[ { M^{d+1}} +  M^d \right]
 + \frac{1}{(2\pi)^{d}}\int_{|u|>M} (|\Delta_{p}\phi_{\xi}(u)|+|\Delta_{p}\phi_{\eta}(u)|)\, du\\  &\leq& \nonumber
 \frac{C^{\circ}_p A d \pi^{d/2}}{(2\pi)^d \Gamma\left(\frac{d}{2} + 1\right)} \left[ M^{d+1} + M^d \right] + 
\frac{2d d_{p,l} \gamma_l}{(2\pi)^d M^{l - d}}.
\end{eqnarray}
by \eqref{fabol}, \eqref{kiralyfi} and \eqref{explin}, respectively. Noting that $|x|^{p}\leq h_{p}\sum_{j=1}^{d}|x_{j}|^{p}$
for a suitable constant $h_{p}$ the following bound can be obtained in the case $M\geq 1$:
\begin{eqnarray}\nonumber
& & |f_{\xi}(x)-f_{\eta}(x)| |x|^{p}  \\ &\leq& \nonumber
\frac{h_pd}{(2\pi)^d} \left[
\frac{C^{\circ}_p A  \pi^{d/2}}{\Gamma\left( \frac{d}{2} + 1 \right)} \left( M^{d+1} +  M^{d+1} \right)
	+ \frac{ 2b_{p,l} \gamma_l}{ M^{l - d}}
\right]
 \\ &\leq& \hat{C}_{l,p} A^{1 - \frac{d+1}{l+1}},\label{mordor}
\end{eqnarray}
by choosing $M := A^{-1/(l+1)}$ and setting the constant properly as 
\begin{equation*}
  \hat{C}_{l,p} := \frac{2h_p d}{(2\pi)^d} \left(
\frac{C^{\circ}_{p} \pi^{d/2}}{ \Gamma\left( \frac{d}{2} + 1 \right)} + b_{p,l} \gamma_l
\right)
\end{equation*}
Note that, choosing $l$ large enough, \eqref{infi} for $|\alpha|=0$ follows from \eqref{mordor}.

Again, fixing $M\geq 1$, $l> d$ we can estimate as follows, using Cauchy's and Markov's inequalities:
\begin{eqnarray*}& & 
\int_{\mathbb{R}^{d}} |f_{\xi}(x)-f_{\eta}(x)||x|^{p}\, dx \\
&\leq & \int_{|x|>M} [f_{\xi}(x)+f_{\eta}(x)]|x|^{p}\, dx + \int_{|x|\leq M} [f_{\xi}(x)-f_{\eta}(x)]|x|^{p}\, dx\\ 
&\leq& E[|\xi|^{p}1_{\{|\xi|>M\}}+|\eta|^{p}1_{\{|\xi|>M\}}]+ \hat{C}_{l,p}  \frac{\pi^{d/2}}{\Gamma\left( \frac{d}{2} + 1 \right)} M^{d+p}  
A^{1 - \frac{d+1}{l+1}} \\
&\leq& E^{1/2}[|\xi|^{2p}]P^{1/2}(|\xi|>M)+E^{1/2}[|\eta|^{2p}]P^{1/2}(|\eta|>M)+\hat{C}_{l,p}  \frac{\pi^{d/2}}{\Gamma\left( \frac{d}{2} + 1 \right)} 
M^{d+p}  A^{1 - \frac{d+1}{l+1}} \\
&\leq& 
 \frac{2(a_{0,2p}a_{0,2l})^{1/2}}{M^{l}}  +\hat{C}_{l,p}  
 \frac{\pi^{d/2}}{\Gamma\left( \frac{d}{2} + 1 \right)} M^{d+p}  A^{1 - \frac{d+1}{l+1}}  
\end{eqnarray*}
whence, choosing $M := A^{- \frac{l - d}{(l+1)(l + p+d)} }$,
$\theta_{l,p} := \frac{(l - d)l}{(l+1)(l + p+d)}$,
$$
\int_{\mathbb{R}^{d}} |f_{\xi}(x)-f_{\eta}(x)||x|^{p}\, dx \leq \bar{C}_{l,p}[M^{-l}+M^{d+p}A^{(1 - \frac{d+1}{l+1})}]\leq 
2\bar{C}_{l,p}A^{\theta_{l,p}}
$$
for a constant 
\begin{equation*}
\bar{C}_{l,p} :=
\frac{\hat{C}_{l,p} \pi^{d/2} }{ \Gamma\left( \frac{d}{2} + 1 \right) } +2{(a_{0,2p} a_{0,l})^{1/2} }.
\end{equation*}

Notice that the above argument can also be performed for $p=0$ and leads to
$$
\int_{\mathbb{R}^{d}} |f_{\xi}(x)-f_{\eta}(x)|\, dx\leq \bar{C}_{l,0}A^{\theta_{l,0}}
$$
hence
\begin{equation*}
\rho_p(\mathcal{L}(\xi), \mathcal{L}(\eta)) \leq 2\bar{C}_{l,p} A^{\theta_{l,p}}+\bar{C}_{l,0}A^{\theta_{l,0}}\leq (2\bar{C}_{l,p}+
\bar{C}_{l,0})A^{\theta_{l,p}}
\end{equation*}
follows. For any $\epsilon \in (0,1)$, choose 
\begin{equation*}
l := \left\lceil \frac{d + \sqrt{1-\epsilon}(p+d)}{1-\sqrt{1-\epsilon}}\right\rceil.
\end{equation*}
Since 
$$
\theta_{l,p}=\frac{(l-d)l}{(l+p+d)(l+1)}\geq \frac{(l-d)^{2}}{(l+p+d)^{2}}\geq 1-\epsilon
$$ 
with this choice of $l$, \eqref{weighted} follows for $0<A\leq 1$. To see \eqref{infi} in the case $|\alpha|\neq 0$,
notice that, from the proof of Proposition \ref{equiv1},
$$
(-i)^{|\alpha|}\partial_{\alpha}f_{\xi}(x)\sum_{j=1}^{d} |x_{j}|^{p}=
\frac{(-i)^{p}}{(2\pi)^{d}}\int_{\mathbb{R}^{d}}u_{1}^{\alpha_{1}}\ldots u_{d}^{\alpha_{d}}\Delta_{p}\phi_{\xi}(u)e^{-i\langle u,x\rangle}\, du,
$$ 
and similarly for $f_{\eta}$. We may thus write
\begin{eqnarray}\label{majkell2}
& & |\partial_{\alpha}f_{\xi}(x)-\partial_{\alpha}f_{\eta}(x)| \left[\sum_{j=1}^{d}|x_{j}|^{p}\right] 
\\ &\leq&\nonumber \frac{1}{(2\pi)^{d}}\int_{|u|\leq M}|\Delta_{p}\phi_{\xi}(u)-\Delta_{p}\phi_{\eta}(u)||u|^{|\alpha|}\, du
+\frac{1}{(2\pi)^{d}}\int_{|u|>M} (|\Delta_{p}\phi_{\xi}(u)|+|\Delta_{p}\phi_{\eta}(u)|)|u|^{|\alpha|}\, du\\  &\leq& \nonumber
C\left[ { AM^{d+1+|\alpha|}} + 1/M^{l-d-|\alpha|} \right],
\end{eqnarray}
which, choosing $M:=A^{-\frac{1}{l+1}}$, shows
$$
|\partial_{\alpha}f_{\xi}(x)-\partial_{\alpha}f_{\eta}(x)| \left[\sum_{j=1}^{d}|x_{j}|^{p}\right]\leq A^{(l-d-|\alpha|)/(l+1)},
$$
which, for $l$ large enough, implies \eqref{infi}. 

If $A>1$ then the previous arguments show that $\rho_{p}(\xi,\eta),||(f_{\xi}(x)-f_{\eta}(x))|x|^{p}||_{\infty}$
are bounded by constants $C,C'$, we omit the trivial but tiresome details. Clearly, $C\leq CA$, $C'\leq C' A$
so \eqref{weighted} as well as \eqref{infi} follow in this case, too. 
\end{proof}

\begin{proof}[Proof of Lemma \ref{main2}] For simplicity, we are less careful with constants in this proof than in the previous one.
Clearly, \eqref{intim2} implies, by Proposition \ref{equiv1}, that the hypotheses of 
Lemma \ref{fourier1} hold hence the arguments of the previous proof are applicable. We assume first that $0<A<e^{-r_{p}}$.
Instead of \eqref{majkell} we can write, for $M\geq 1$
\begin{eqnarray*}\nonumber
& & |f_{\xi}(x)-f_{\eta}(x)| \left[\sum_{j=1}^{d}|x_{j}|^{p}\right] 
\\ &\leq&\nonumber \int_{|u|\leq M}|\Delta_{p}\phi_{\xi}(u)-\Delta_{p}\phi_{\eta}(u)|\, du
+\int_{|u|>M} (|\Delta_{p}\phi_{\xi}(u)|+|\Delta_{p}\phi_{\eta}(u)|)\, du\\  &\leq& \nonumber
\frac{2M^{d+1}C^{\circ}_p A d \pi^{d/2}}{\Gamma\left(\frac{d}{2} + 1\right)} 
+ \int_{|u|>M} (|\Delta_{p}\phi_{\xi}(u)|+|\Delta_{p}\phi_{\eta}(u)|)e^{r_{p}|u|/2}e^{-r_{p}|u|/2}\, du\\  &\leq& \nonumber
\frac{2M^{d+1}C^{\circ}_p A d \pi^{d/2}}{\Gamma\left(\frac{d}{2} + 1\right)}\\ 
&+& 
\left(\int_{\mathbb{R}^{d}} (|\Delta_{p}\phi_{\xi}(u)|+|\Delta_{p}\phi_{\eta}(u)|)^{2}e^{r_{p}|u|}\,du{}\right)^{1/2}
\left(\int_{|u|\geq M}e^{-r_{p}|u|}\, du\right)^{1/2}\\
&\leq& C\left[ M^{d+1}A+ \left(\int_{M}^{\infty}e^{-r_{p}y}y^{d-1}\, dy\right)^{1/2}\right]\\
&\leq& C'[M^{d+1}+e^{-r_{p}M/2}M^{(d-1)/2}],
\end{eqnarray*}
with constants $C,C'$ so choosing $M:=2\ln|A|/r_{p}$ we arrive at 
$$
|f_{\xi}(x)-f_{\eta}(x)| \left[\sum_{j=1}^{d}|x_{j}|^{p}\right]\leq C'' A|\ln A|^{d+1}.
$$
Arguing further with Markov's inequality,
\begin{eqnarray*}
& & \int_{\mathbb{R}^{d}}|f_{\xi}(x)-f_{\eta}(x)||x|^{p}\, dx\\
&\leq& \int_{|x|\leq M}C''A |\ln A|^{d+1}\, dx +\int_{|x|>M} [f_{\xi}(x)+f_{\eta}(x)]\, dx\\
&\leq& C'''M^{d}A |\ln A|^{d+1}+\left[E[e^{r|\xi|}+e^{r|\eta|}]\right]e^{-rM}\\
&\leq& C'''M^{d}A |\ln A|^{d+1}+2C^{\sharp}e^{-rM}.	
\end{eqnarray*} 
Choosing $M:=|\ln A|/r$, we arrive at
$$
\int_{\mathbb{R}^{d}}|f_{1}(x)-f_{2}(x)||x|^{p}\, dx\leq C'''' A|\ln A|^{2d+1}.
$$
This follows easily for $p=0$ as well, so we can conclude for the present case.
The case $A>e^{-r_{p}}$ is also easily handled, we omit details.
\end{proof}

\begin{remark}{\rm It is desirable to find a sufficient condition for \eqref{intim2} 
in terms of densities only.
For simplicity, let $d=1$. Let $\mathbf{F}$ be a collection of holomorphic functions on the strip
$$
\{z\in\mathbb{C}: z=a+bi,\ a\in\mathbb{R},\ -\varrho<b<\varrho\}
$$
for some $\varrho>0$ such that the restriction of each $f\in\mathbf{F}$ is a density on $\mathbb{R}$ and let
\begin{equation}\label{suppy}
\sup_{f\in\mathbb{F}}\sup_{-\varrho_{n}<b<\varrho_{n}} \int_{\mathbb{R}}x^{2k}f^{2}(a+bi)\, da<C_{k}
\end{equation}
hold for all $k\in\mathbb{N}$ with some (finite) $C_{k}$ and $0<\varrho_{k}\leq \varrho$. We claim that, for each $k$, 
$$
\sup_{f\in\mathbf{F}}\int_{\mathbb{R}^{d}}|\phi_{f}^{(k)}(u)|e^{r_{k}|u|}\leq c_{k}
$$
holds in this case for suitable $r_{k},c_{k}>0$, here $\phi_{f}$ is the characteristic function corresponding to $f$. 

We only sketch this for $k=0$, the general case being analogous. Recall the relevant version of the Paley-Wiener theorem (Theorem IV of
\cite{pw}, note that this is stated for the \emph{inverse} Fourier transform according to our terminology,
but it nevertheless holds true also for the Fourier transform) which implies that  
$$
\sup_{f\in\mathbf{F}}\int_{\mathbb{R}^{d}}|\phi_{f}(u)|^{2}e^{\bar{r}|u|}\, du\leq C_{\bar{r}}
$$
holds for suitable $\bar{r}$, $C_{\bar{r}}>0$. The Cauchy inequality guarantees
\begin{eqnarray*}
& & \int_{\mathbb{R}^{d}}|\phi_{f}(u)|e^{\bar{r}|u|/4}\, du\\
&=& \int_{\mathbb{R}^{d}}|\phi_{f}(u)|e^{\bar{r}|u|/2}e^{-\bar{r}|u|/4}\, du\\
&\leq& \left(\int_{\mathbb{R}^{d}}|\phi_{f}(u)|^{2}e^{\bar{r}|u|}\, du\right)^{1/2}
\left(\int_{\mathbb{R}^{d}}e^{-\bar{r}|u|/2}\, du\right)^{1/2}
\end{eqnarray*}
so we may choose $r_{0}:=\bar{r}/4$ and
$$
c_{0}:=\sqrt{C_{\bar{r}}}\left(\int_{\mathbb{R}^{d}}e^{-\bar{r}|u|/2}\, du\right)^{1/2}<\infty.
$$
}
\end{remark}

\section{Convergence of functionals in Malliavin calculus}\label{sec4}

This section recalls the necessary analytical framework of Malliavin calculus required 
for the derivation of our main regularity and approximation results. 
Our approach follows the standard constructions presented in \cite[Chapter~2]{nualart} and \cite[Section~2.1]{bally}.

\subsubsection*{Canonical Wiener Space}

Let \( m \in \mathbb{N} \) be fixed. We define the \emph{canonical Wiener space} as the triplet
\[
(\Omega, \mathcal{F}, P) := \left(C_0([0,1]; \mathbb{R}^m), \mathcal{B}, P\right),
\]
where:
\begin{itemize}
    \item \( \Omega := C_0([0,1]; \mathbb{R}^m) \) is the Banach space of continuous paths \(\omega: [0,1] \to \mathbb{R}^m\) such that \(\omega(0) = 0\), equipped with the supremum norm;
    \item \( \mathcal{B}\) is the Borel \(\sigma\)-algebra on \(\Omega\);
    \item \( P \) is the \emph{Wiener measure} on \((\Omega, \mathcal{B})\), under which the coordinate process defined by
    \[
    B_t(\omega) := \omega(t), \quad \omega \in \Omega,\ t \in [0,1],
    \]
    is a standard \(m\)-dimensional Brownian motion.
\end{itemize}

\subsubsection*{Cameron--Martin Space}

The \emph{Cameron--Martin space} \(\mathcal{H}\) associated to this probability space is the Hilbert space defined by:
\[
\mathcal{H} := \left\{ h \in C([0,1]; \mathbb{R}^m): h(t) = \int_0^t \dot{h}_s \, ds,\ \dot{h} \in L^2([0,1]; \mathbb{R}^m) \right\}.
\]
Here, for \( h \in \mathcal{H} \), the notation \( \dot{h} \) means that $h$ is absolutely conrinuous
with (Lebesgue-a.e.) derivative $\dot{h}$.

We equip \( \mathcal{H} \) with the inner product:
\[
\langle h, k \rangle_{\mathcal{H}} := \int_0^1 \langle \dot{h}_s, \dot{k}_s \rangle_{\mathbb{R}^m} \, ds,
\]
and the corresponding norm \( \|h\|_{\mathcal{H}} := \sqrt{\langle h, h \rangle_{\mathcal{H}}} \).

For each \( h \in \mathcal{H} \), the \emph{Wiener integral} of \( h \) is defined as
    \[
    W(h) := \int_0^1 \langle \dot{h}_t, dB_t \rangle_{\mathbb{R}^m}.
    \]

\subsubsection*{The Malliavin Derivative and Sobolev Spaces $\mathbb{D}^{k,p}$}

Let \( n \in \mathbb{N} \), and let \( h_1, \ldots, h_n \in \mathcal{H} \) be fixed. Let \( f \in C_b^\infty(\mathbb{R}^n) \), the space of smooth functions with bounded partial derivatives of all orders. Then we define a \emph{smooth cylindrical functional} 
\( F: \Omega \to \mathbb{R} \) by
\[
F(\omega) := f(W(h_1), \ldots, W(h_n)), \quad \omega \in \Omega.
\]
The class of all such functionals is denoted by $\mathcal{S}$,
they form the basic class on which the Malliavin derivative is initially defined and they are dense in \( L^p(\Omega) \) for all \( p \geq 1 \).

For $F\in\mathcal{S}$, its \emph{Malliavin derivative} \( D_t F \in \mathbb{R}^m \) is defined for \( t \in [0,1] \) by
\[
D_t F := \sum_{i=1}^n \partial_i f(W(h_1), \ldots, W(h_n)) \cdot \dot{h}_i(t),
\]
where \( \partial_i f \) denotes the partial derivative of \( f \) with respect to its \( i \)-th coordinate.

This definition provides a process \( D F := (D_t F)_{t \in [0,1]} \in L^2([0,1]; \mathbb{R}^m) \), and may be 
interpreted as the derivative of \( F \) along infinitesimal shifts in the direction of the Brownian motion.

Let \( 1\leq p<\infty\). A random variable \( F \in L^p(\Omega) \) is said to belong to the Malliavin--Sobolev space \( \mathbb{D}^{1,p} \) 
if there exists a stochastic process \( D_t F \in L^p(\Omega \times [0,1]; \mathbb{R}^m) \) such that for all \( h \in \mathcal{H} \),
\[
\mathbb{E}[\partial_h F] = \mathbb{E}[\langle D F, h \rangle_{\mathcal{H}}].
\]
Alternatively, the derivative operator defined on $\mathcal{S}$ is closable, and $\mathbb{D}^{1,p}$ is the domain of its closure,
see \cite{nualart} for details.
We equip \( \mathbb{D}^{1,p} \) with the norm
\[
\|F\|_{1,p}^p := \mathbb{E}[|F|^p] + \mathbb{E}\left[ \left( \int_0^1 |D_t F|^2 \, dt \right)^{p/2} \right],
\]
and denote by \( \mathbb{D}^{1,p} \) the closure of the smooth cylindrical functionals under this norm.

Let \( k \geq 1 \) be an integer. The \( k \)-th order Malliavin derivative \( D^k F \) is defined recursively by
\[
D^k F := D(D^{k-1} F),
\]
with \( D^0 F := F \). For each \( k \), \( D^k F \) is a random field indexed by \( [0,1]^k \), 
taking values in the tensor product space \( (\mathbb{R}^m)^{\otimes k} \).

We define the Malliavin--Sobolev space \( \mathbb{D}^{k,p} \) to consist of all \( F \in L^p(\Omega) \) such that
\[
D^j F \in L^p\left(\Omega; L^2([0,1]^j; (\mathbb{R}^m)^{\otimes j}) \right), \quad \text{for all } 1 \leq j \leq k.
\]
The norm on \( \mathbb{D}^{k,p} \) is given by
\[
\|F\|_{k,p}^p := \mathbb{E}[|F|^p] + \sum_{j=1}^k \mathbb{E}\left[ \left( \int_{[0,1]^j} |D^j_{t_1, \ldots, t_j} F|^2 dt_1 \cdots dt_j \right)^{p/2} \right].
\]

We define
\[
\mathbb{D}^\infty := \bigcap_{k \geq 1} \bigcap_{p \geq 1} \mathbb{D}^{k,p}.
\]

For instance, solutions of stochastic differential equations with smooth coefficients lie in $\mathbb{D}^{\infty}$, see Section 2.2 of
\cite{nualart}. 

\subsubsection*{Non-Degeneracy and Existence of Smooth Densities}

Let \( F = (F^1, \ldots, F^d) \in (\mathbb{D}^{1,2})^d \) be a random vector defined on the Wiener space \( (\Omega, \mathcal{F}, {P}) \). 
The \emph{Malliavin covariance matrix} (also called the \emph{Malliavin matrix}) associated to \( F \) is the random \( d \times d \) symmetric matrix
\[
\sigma_F := \left( \langle D F^i, D F^j \rangle_{\mathcal{H}} \right)_{1 \leq i,j \leq d} = \left( \int_0^1 \langle D_t F^i, D_t F^j \rangle_{\mathbb{R}^m} dt \right)_{1 \leq i,j \leq d}.
\]

We say that the random vector \( F \in (\mathbb{D}^{1,2})^d \) is \emph{non-degenerate} in the sense of Malliavin calculus if:
\begin{equation}\label{eq:nondeg}
\forall r > 0, \quad \mathbb{E}\left[ |\det \sigma_F|^{-r} \right] < \infty.
\end{equation}
This condition ensures that \( \sigma_F \) is invertible almost surely and that its 
inverse moments are finite. 

If $F$ is smooth ($F \in (\mathbb{D}^{q,\infty})^d$ for some $q$ large enough) and non-degenerate then Malliavin calculus allows
to show that it has a (smooth) density. Tail estimates (from above as well as from below) can also be given for the density,
see \cite{nualart, bally}. We recall one concrete result here, \cite[Theorem~2.3.1]{bally}.

\begin{theorem}\label{thm:bally-caramellino}
Let \( q \in \mathbb{N} \), and suppose that \( F \in (\mathbb{D}^{q+2,\infty})^d \) satisfies the non-degeneracy condition \eqref{eq:nondeg}. Then:
\begin{enumerate}
    \item The law of \( F \) has a $q$ times continuously differentiable density \( f_F  \).
    \item For every multi-index \( \alpha \in \mathbb{N}^d \) with 
    \( |\alpha| \leq q \) and for all $l\in\mathbb{N}$, there exists a constant \( C > 0 \) 
    depending on \( \|F\|_{q+2,p} \), \( \|\det \sigma_F^{-1}\|_{r} \), $l$ 
    and \( |\alpha| \), such that
    \[
    |\partial_{\alpha} f_F(x)| \leq \frac{C}{(1 + |x|)^l}, 
    \quad \text{for all } x \in \mathbb{R}^d.
    \]
\end{enumerate}
\end{theorem}

\begin{remark}
{\rm The constants in the above bounds can be made explicit in terms of norms \( \|F\|_{k,p} \), inverse moments 
\( \|\det \sigma_F^{-1}\|_r \), and dimension-dependent combinatorics (cf. \cite[Eq.~(2.88), (2.89)]{bally}). The decay estimate
\[
|\partial_{\alpha} f_F(x)| \leq \frac{C}{(1 + |x|)^l}
\]
is particularly important in distributional approximations, as it guarantees integrability and localization of the density and its derivatives.}
\end{remark}

\subsubsection*{Convergence of smooth, non-degenerate functionals}

The conclusion of the next proposition -- that each functional \( F_i \in (\mathbb{D}^{q+2,\infty})^d \) possesses a 
\( C^\infty \)-density \( f_i \) with globally controlled polynomially weighted derivatives -- follows as a 
direct consequence of Theorem~\ref{thm:bally-caramellino} above.

\begin{proposition}\label{malliavin-estimate} Let $I$ be an arbitrary
index set. Let $q\geq 1$ be an integer and
let $F_{i}\in (\mathbb{D}^{q+2,\infty})^{d}$, $i\in I$ be a family of functionals
with corresponding Malliavin matrices $\sigma_{i}$.
Let 
$$
C_{r}^{\sigma}:=\sup_{i\in I}E[|\mathrm{det}(\sigma_{i})|^{-r}]<\infty,
$$
as well as
$$
C_{r}^{F}:=\sup_{i\in I}\Vert F_{i}\Vert_{q+2,r}<\infty,
$$
hold for all $r\geq 1$. Then $F_{i}$ has a continuous density $f_{i}$
with respect to the $d$-dimensional Lebesgue measure,
and for all $l\geq 1$ we have
\begin{equation}\label{kelli}
\max_{|\alpha|\leq q}\sup_{i\in I}\sup_{x\in\mathbb{R}^{d}}|\partial_{\alpha} f_{i}(x)|(1+|x|)^{l}\leq C_{q,l}^{F}
\end{equation}
for some finite $C_{q,l}^{F}$.
\end{proposition}
\begin{proof}
See Proposition 5.1 in \cite{rate}. 	
\end{proof}

The next result follows from Lemma 3.9 and Proposition 3.15 in \cite{bc2}, see
also \cite{bc1} for earlier related results. 
Recall that, for arbitrary $\mu_{1},\mu_{2}\in\mathcal{P}$ their \emph{Fortet-Mourier distance} is
defined by 
$$
d_{FM}(\mu_{1},\mu_{2}):=\sup_{f\in\mathcal{M}_{1}}\left|\int_{\mathbb{R}^{d}}f(x)\mu_{1}(dx)-\int_{\mathbb{R}^{d}}f(x)\mu_{2}(dx)\right|,
$$
where $\mathcal{M}_{1}$ is the set of continuously differentiable functions $f:\mathbb{R}^{d}\to\mathbb{R}$
such that $||f||_{\infty}\leq 1$ and $||\nabla f||_{\infty}\leq 1$. Clearly, this is a weaker metric than $W_{1}$.

\begin{theorem}\label{smoothsequence11} 
Let $F,F_{n}\in (\mathbb{D}^{\infty,\infty})^{d}$, $n\in \mathbb{N}$ be a family of functionals
with corresponding Malliavin matrices $\sigma,\sigma_{n}$.
Assume that, for all $r>0$, $q\geq 1$,
$$
E[|\mathrm{det}(\sigma)|^{-r}]+\sup_{n\in \mathbb{N}}E[|\mathrm{det}(\sigma_{n})|^{-r}]<\infty,
$$
and also 
$$
\Vert F\Vert_{q,r}+\sup_{n\in\mathbb{N}}\Vert F_{n}\Vert_{q,r}<\infty.
$$	
Then for all integers $k\geq 1$ and for all $\epsilon>0$ there is a constant $C_{\epsilon,k}$ such that
$$
d_{TV}(\mathcal{L}(F_{n}),\mathcal{L}(F))\leq C_{\epsilon,k}d_{FM}^{1-\epsilon}(\mathcal{L}(F_{n}),\mathcal{L}(F)).
$$
Furthermore, for the respective densities $f,f_{n}$, for all multiindices $\alpha$,
$$
\sup_{x\in\mathbb{R}^{d}}|\partial_{\alpha}f_{n}(x)-\partial_{\alpha}f(x)|\leq C_{\epsilon,k,\alpha}
d_{FM}^{1-\epsilon}(\mathcal{L}(F_{n}),\mathcal{L}(F)).
$$
\hfill $\Box$
\end{theorem}

Now we present the promised application of Lemma \ref{fourier1} to Malliavin calculus, this provides estimates
in \emph{weighted} total variation norm, in terms of Wasserstein distances.

\begin{theorem}\label{smoothsequence2} Let the conditions of Theorem \ref{smoothsequence11} above hold. 
Then, for all $p \geq 1$, $q > 1$, and $\epsilon > 0$, there exists a constant $C_{\epsilon, p, q}$ such that
\[
\rho_p(\mathcal{L}(F_n),\mathcal{L}(F)) \leq C_{\epsilon, p, q}  W_q(\mathcal{L}(F_n),\mathcal{L}(F))^{(1 - \epsilon)}.
\]
\end{theorem}

\begin{proof}
Fix $p \geq 1$, $q > 1$, and $\epsilon \in (0,1)$. By the assumptions of the theorem and Proposition \ref{malliavin-estimate}, the random vectors $F_n$ and $F$ admit $C^\infty$ densities $f_n$, $f$, respectively, and for each $k, l \in \mathbb{N}$,
\[
\sup_{n \in \mathbb{N}} \sup_{x \in \mathbb{R}^d} |\partial^{\alpha} f_n(x)| (1 + |x|)^l < \infty, \quad
\sup_{x \in \mathbb{R}^d} |\partial^{\alpha} f(x)| (1 + |x|)^l < \infty
\]
for all multi-indices $\alpha$ with $|\alpha| \leq k$. In particular, condition 
\eqref{intim} in Lemma \ref{fourier1} is satisfied with uniform constants $d_{k,l}$. Applying Lemma \ref{fourier1} to the pair $(F_n, F)$ yields
\[
\rho_p(\mathcal{L}(F_n),\mathcal{L}(F)) \leq C W_q(\mathcal{L}(F_n),\mathcal{L}(F))^{(1 - \epsilon)},
\]
where the constant $C$ depends only on $p$, $q$, $\epsilon$, and the sequence $(d_{k,l})$, hence it is uniform in $n$. This proves the claim.
\end{proof}




\begin{thebibliography}{14}
\ifx \bisbn   \undefined \def \bisbn  #1{ISBN #1}\fi
\ifx \binits  \undefined \def \binits#1{#1}\fi
\ifx \bauthor  \undefined \def \bauthor#1{#1}\fi
\ifx \batitle  \undefined \def \batitle#1{#1}\fi
\ifx \bjtitle  \undefined \def \bjtitle#1{#1}\fi
\ifx \bvolume  \undefined \def \bvolume#1{\textbf{#1}}\fi
\ifx \byear  \undefined \def \byear#1{#1}\fi
\ifx \bissue  \undefined \def \bissue#1{#1}\fi
\ifx \bfpage  \undefined \def \bfpage#1{#1}\fi
\ifx \blpage  \undefined \def \blpage #1{#1}\fi
\ifx \burl  \undefined \def \burl#1{\textsf{#1}}\fi
\ifx \doiurl  \undefined \def \doiurl#1{\url{https://doi.org/#1}}\fi
\ifx \betal  \undefined \def \betal{\textit{et al.}}\fi
\ifx \binstitute  \undefined \def \binstitute#1{#1}\fi
\ifx \binstitutionaled  \undefined \def \binstitutionaled#1{#1}\fi
\ifx \bctitle  \undefined \def \bctitle#1{#1}\fi
\ifx \beditor  \undefined \def \beditor#1{#1}\fi
\ifx \bpublisher  \undefined \def \bpublisher#1{#1}\fi
\ifx \bbtitle  \undefined \def \bbtitle#1{#1}\fi
\ifx \bedition  \undefined \def \bedition#1{#1}\fi
\ifx \bseriesno  \undefined \def \bseriesno#1{#1}\fi
\ifx \blocation  \undefined \def \blocation#1{#1}\fi
\ifx \bsertitle  \undefined \def \bsertitle#1{#1}\fi
\ifx \bsnm \undefined \def \bsnm#1{#1}\fi
\ifx \bsuffix \undefined \def \bsuffix#1{#1}\fi
\ifx \bparticle \undefined \def \bparticle#1{#1}\fi
\ifx \barticle \undefined \def \barticle#1{#1}\fi
\bibcommenthead
\ifx \bconfdate \undefined \def \bconfdate #1{#1}\fi
\ifx \botherref \undefined \def \botherref #1{#1}\fi
\ifx \url \undefined \def \url#1{\textsf{#1}}\fi
\ifx \bchapter \undefined \def \bchapter#1{#1}\fi
\ifx \bbook \undefined \def \bbook#1{#1}\fi
\ifx \bcomment \undefined \def \bcomment#1{#1}\fi
\ifx \oauthor \undefined \def \oauthor#1{#1}\fi
\ifx \citeauthoryear \undefined \def \citeauthoryear#1{#1}\fi
\ifx \endbibitem  \undefined \def \endbibitem {}\fi
\ifx \bconflocation  \undefined \def \bconflocation#1{#1}\fi
\ifx \arxivurl  \undefined \def \arxivurl#1{\textsf{#1}}\fi
\csname PreBibitemsHook\endcsname

\bibitem[\protect\citeauthoryear{Douc et~al.}{2008}]{eric}
\begin{bbook}
\bauthor{\bsnm{Douc}, \binits{R.}},
\bauthor{\bsnm{Moulines}, \binits{E.}},
\bauthor{\bsnm{Priouret}, \binits{P.}},
\bauthor{\bsnm{Soulier}, \binits{P.}}:
\bbtitle{{M}arkov Chains: Basic Definitions}.
\bpublisher{Springer},
\blocation{Berlin}
(\byear{2008})
\end{bbook}
\endbibitem

\bibitem[\protect\citeauthoryear{Meyn and Tweedie}{2009}]{meyn-tweedie}
\begin{bbook}
\bauthor{\bsnm{Meyn}, \binits{S.}},
\bauthor{\bsnm{Tweedie}, \binits{R.L.}}:
\bbtitle{Markov Chains and Stochastic Stability. 2nd Edition}.
\bpublisher{Cambridge University Press},
\blocation{Cambridge}
(\byear{2009})
\end{bbook}
\endbibitem

\bibitem[\protect\citeauthoryear{Hairer and Mattingly}{2011}]{hm}
\begin{bchapter}
\bauthor{\bsnm{Hairer}, \binits{M.}},
\bauthor{\bsnm{Mattingly}, \binits{J.}}:
\bctitle{Yet another look at {H}arris' ergodic theorem for {M}arkov chains}.
In: \bbtitle{In: Seminar on Stochastic Analysis, Random Fields and Applications
  VI. (eds. R. Dalang, M. Dozzi and F. Russo F.)},
vol. \bseriesno{63},
pp. \bfpage{109}--\blpage{117}.
\bpublisher{Springer},
\blocation{Basel}
(\byear{2011})
\end{bchapter}
\endbibitem

\bibitem[\protect\citeauthoryear{Villani}{2009}]{villani}
\begin{bbook}
\bauthor{\bsnm{Villani}, \binits{C.}}:
\bbtitle{Optimal Transport: Old and New}.
\bpublisher{Springer},
\blocation{Berlin}
(\byear{2009})
\end{bbook}
\endbibitem

\bibitem[\protect\citeauthoryear{Eberle et~al.}{2019}]{zimmer}
\begin{botherref}
\oauthor{\bsnm{Eberle}, \binits{A.}},
\oauthor{\bsnm{Guillin}, \binits{A.}},
\oauthor{\bsnm{Zimmer}, \binits{R.}}:
Quantitative {H}arris-type theorems for diffusions and {M}c{K}ean--{V}lasov
  processes.
Transactions of the American Mathematical Society
(370),
7135--7173
(2019)
\end{botherref}
\endbibitem

\bibitem[\protect\citeauthoryear{Eberle and Majka}{2019}]{eberle}
\begin{barticle}
\bauthor{\bsnm{Eberle}, \binits{A.}},
\bauthor{\bsnm{Majka}, \binits{M.B.}}:
\batitle{Quantitative contraction rates for {M}arkov chains on general state
  spaces}.
\bjtitle{Electronic Journal of Probability}
\bvolume{24}(\bissue{26}),
\bfpage{1}--\blpage{36}
(\byear{2019})
\end{barticle}
\endbibitem

\bibitem[\protect\citeauthoryear{Majka et~al.}{2020}]{majka}
\begin{botherref}
\oauthor{\bsnm{Majka}, \binits{M.B.}},
\oauthor{\bsnm{Mijatovi\'c}, \binits{A.}},
\oauthor{\bsnm{Szpruch}, \binits{L.}}:
Non-asymptotic bounds for sampling algorithms without log-concavity.
Annals of Applied Probability
(30),
1534--1581
(2020)
\end{botherref}
\endbibitem

\bibitem[\protect\citeauthoryear{R\'asonyi.}{2024}]{rate}
\begin{botherref}
\oauthor{\bsnm{R\'asonyi.}, \binits{M.}}:
Rate estimates for total variation norm with applications.
Submitted
(arXiv:2412.04013)
(2024)
\end{botherref}
\endbibitem

\bibitem[\protect\citeauthoryear{Nualart}{2006}]{nualart}
\begin{bbook}
\bauthor{\bsnm{Nualart}, \binits{D.}}:
\bbtitle{The {M}alliavin Calculus and Related Topics}.
\bpublisher{Springer},
\blocation{Berlin}
(\byear{2006})
\end{bbook}
\endbibitem

\bibitem[\protect\citeauthoryear{Bally and Caramellino}{2016}]{bally}
\begin{botherref}
\oauthor{\bsnm{Bally}, \binits{V.}},
\oauthor{\bsnm{Caramellino}, \binits{L.}}:
Integration by parts formulas, {M}alliavin calculus and regularity of
  probability laws.
V. Bally, L. Caramellino and R. Cont: Stochastic integration by parts and
  functional It{\^o} calculus,
1--114
(2016)
\end{botherref}
\endbibitem

\bibitem[\protect\citeauthoryear{Bally and Caramellino}{2014}]{bc1}
\begin{barticle}
\bauthor{\bsnm{Bally}, \binits{V.}},
\bauthor{\bsnm{Caramellino}, \binits{L.}}:
\batitle{On the distances between probability density functions}.
\bjtitle{Electronic Journal of Probability}
\bvolume{19}(\bissue{110}),
\bfpage{1}--\blpage{33}
(\byear{2014})
\end{barticle}
\endbibitem

\bibitem[\protect\citeauthoryear{Bally et~al.}{2020}]{bc2}
\begin{barticle}
\bauthor{\bsnm{Bally}, \binits{V.}},
\bauthor{\bsnm{Caramellino}, \binits{L.}},
\bauthor{\bsnm{Poly}, \binits{G.}}:
\batitle{Regularization lemmas and convergence in total variation}.
\bjtitle{Electronic Journal of Probability}
\bvolume{25}(\bissue{74}),
\bfpage{1}--\blpage{20}
(\byear{2020})
\end{barticle}
\endbibitem

\bibitem[\protect\citeauthoryear{Stein and Weiss}{1971}]{fourier}
\begin{bbook}
\bauthor{\bsnm{Stein}, \binits{E.M.}},
\bauthor{\bsnm{Weiss}, \binits{G.}}:
\bbtitle{Introduction to {F}ourier Analysis on {E}uclidean Spaces}.
\bpublisher{Princeton University Press},
\blocation{Princeton}
(\byear{1971})
\end{bbook}
\endbibitem

\bibitem[\protect\citeauthoryear{Paley and Wiener}{1934}]{pw}
\begin{bbook}
\bauthor{\bsnm{Paley}, \binits{R.C.}},
\bauthor{\bsnm{Wiener}, \binits{N.}}:
\bbtitle{Fourier Transforms in the Complex Domain}.
\bpublisher{American Mathematical Society},
\blocation{New York}
(\byear{1934})
\end{bbook}
\endbibitem

\end{thebibliography}

\end{document}